\documentclass[11pt,reqno]{amsart}
\usepackage[T1]{fontenc}
\usepackage{amsmath,amsthm,amssymb,mathrsfs}
\usepackage{lmodern,microtype,mathtools,bm,enumitem,needspace}
\usepackage{xcolor}
\usepackage{hyperref}
\usepackage{bookmark}
\microtypesetup{expansion=false}
\setlist[itemize]{leftmargin=2.2em,itemsep=0.15em,topsep=0.25em}
\newtheorem{theorem}{Theorem}[section]
\newtheorem{lemma}[theorem]{Lemma}
\newtheorem{proposition}[theorem]{Proposition}
\newtheorem{corollary}[theorem]{Corollary}
\theoremstyle{definition}
\newtheorem*{AIuse}{AI Usage Statement}
\theoremstyle{remark}\newtheorem{remark}[theorem]{Remark}
\usepackage[nameinlink,capitalize,noabbrev]{cleveref}
\newcommand{\Hn}{\mathbb H^n}
\newcommand{\Haus}{\mathcal H}
\newcommand{\LTheta}{{L_\Theta}}
\newcommand{\gOne}{{g_1}}
\newcommand{\dd}{\,\mathrm d}
\newcommand{\spt}{\operatorname{spt}}
\newcommand{\RR}{\mathbb R}
\newcommand{\eps}{\varepsilon}

\newcommand{\divh}{\operatorname{div}_H}
\newcommand{\Ca}{\mathcal C}
\newcommand{\PSphere}{\mathcal S}

\numberwithin{equation}{section}
\allowdisplaybreaks[1]
\title[The Isoperimetric Inequality]{The Isoperimetric Inequality\\
in Higher-Dimensional Heisenberg Groups}
\hypersetup{hypertexnames=false,hidelinks,bookmarksnumbered=true,
  pdftitle={The Isoperimetric Inequality in Higher-Dimensional Heisenberg Groups},
  pdfauthor={Wei-Chun Chen, Wei-Ting Kao, and Chun-Kai Lien}}
\author{Wei-Chun Chen}
\author{Wei-Ting Kao}
\author{Chun-Kai Lien}
\date{}
\begin{document}
\begin{abstract}
We prove the sharp horizontal isoperimetric inequality in
$\Hn$, $n\ge2$, with equality precisely for left translates and
Heisenberg dilations of Pansu balls, up to Haar-null sets.
\end{abstract}
\maketitle
\vspace{-12pt}

% ===== 01-selected-inputs =====
\section{Introduction}\label{sec:target}

The sharp isoperimetric problem asks which sets enclose a prescribed
Haar volume with the least horizontal perimeter. We identify every
minimizer in $\Hn$, $n\ge2$, as a Pansu ball. Write $Q=2n+2$
and $(z,t)\in\RR^{2n}\times\RR$. In the conventions of
Section~\ref{sec:centre-reconstruction}, the Pansu ball of horizontal radius
one centred at the identity and its isoperimetric ratio are
\begin{gather}
 \Omega:=\left\{(z,t):|z|<1,\quad
 2|t|<\arccos|z|+|z|\sqrt{1-|z|^2}\right\},\notag\\
 C_n:=\frac{P_H(\Omega)}
 {|\Omega|^{(Q-1)/Q}}
 =\frac{4n(n+1)}{2n+1}
 \left(\frac{\pi^{n+1}(2n+1)!}
 {2^{2n+1}(n+1)(n!)^3}\right)^{1/(2n+2)}.
 \label{eq:sharp-constant}
\end{gather}
Left translations and Heisenberg dilations preserve this ratio.
The problem in $\mathbb H^1$ goes back to Pansu~\cite{Pansu1982}.

\begin{theorem}[Sharp horizontal isoperimetric inequality]
\ifcsname r@thm:main\endcsname%
\bookmark[level=2,dest={\getrefbykeydefault{thm:main}{anchor}{}}]%
  {Theorem \ref*{thm:main}: Sharp horizontal isoperimetric inequality}%
\fi
Let\phantomsection\label{thm:main} $n\ge2$ and $Q=2n+2$. For every measurable set $E\subset\Hn$
with finite horizontal perimeter and $0<|E|<\infty$,
\begin{equation}\label{eq:sharp-isoperimetric}
 P_H(E)\ge C_n|E|^{(Q-1)/Q}
 =P_H(\Omega)
   \left(\frac{|E|}{|\Omega|}\right)^{(Q-1)/Q}.
\end{equation}
The constant $C_n$ is optimal. Equality holds if and only if, up to
a Haar-null set, $E=a\cdot\delta_r \Omega$ for some
$a\in\Hn$ and $r>0$.
\end{theorem}

Existence of minimizers reduces the inequality and its equality cases
to classifying minimizers, without assuming boundary regularity.

Danielli, Garofalo and Nhieu~\cite{DGNIsoperimetric} prove the sharp
inequality for sets described by upper and lower $C^2$ graphs over
the same horizontal ball, with equal upper and lower volumes,
zero boundary values and disjoint interior zero sets.
Monti~\cite{MontiAxial} classifies minimizers among axially symmetric
finite-perimeter sets. Ritor\'e's calibration theorem~\cite[Theorem~3.1]{Ritore} gives the
sharp comparison and equality classification for finite-perimeter sets
containing a horizontal disk and contained in the vertical cylinder
over that disk. The companion rigidity theorem~\cite[Theorem~1.1]{Rigidity}
classifies volume-preserving stable boundaries of bounded open sets
under a $C^2$ regularity assumption.

For the selected elliptic approximations of a minimizer, the
reconstruction in \cite[Section~2]{Rigidity} associates a candidate
centre to each regular boundary point and its normal data.
This centre is constant on a unit Pansu sphere.
The quantitative estimate of \cite[Theorem~4.1]{CentreRigidity}
makes its horizontal tangential derivative tend to zero
in $L^1$ of horizontal perimeter.

Different components may still carry different centres, but a
global minimizer cannot allocate a closed part of its boundary
independently of the rest. The volume profile gives a Sobolev
inequality attained by $\mathbf1_E$; strict convexity of
$L^{Q/(Q-1)}$ forces every closed part with the same outward
orientation to occupy a constant proportion of the whole boundary.
Proposition~\ref{weak:prop:closed-subcurrent} makes this precise
for closed subcurrents of the horizontal boundary current.
Weights supported near a limiting centre then force the whole
boundary onto the sphere with that centre.

The restriction $n\ge2$ enters in two places. Near characteristic
points it provides two transverse coordinates and hence cutoffs
with vanishing surface Dirichlet energy. In the current argument,
$2n-1\ge n+1$ places the test forms in the upper Rumin complex,
where $d_c$ is the exterior derivative and satisfies the usual
product rule.

Section~\ref{sec:weak-characterization} states the recognition
criterion using \cite{Rigidity}. Section~\ref{sec:approximation}
approximates the minimizer supplied by Leonardi--Rigot~\cite{LR}.
Section~\ref{sel:sec:capacity} constructs cutoffs around singular
and characteristic points and applies \cite[Theorem~4.1]{CentreRigidity}.
Sections~\ref{sec:weighted-currents} and~\ref{sec:one-centre-recognition}
classify the limit.

Unless stated otherwise, constants without a subscript $j$ are
uniform in $j$; constants with that subscript may depend on $j$.

\begin{AIuse}
The authors developed the proof strategy for the $C^2$ setting
in their companion paper on stable rigidity. ChatGPT 6.0 Astra
assisted in identifying relevant references and suitable mathematical
tools for extending this strategy to the more general setting
considered in this paper. The authors personally checked the
references and verified the applicability of the suggested tools.
The mathematical arguments were developed and checked by the
authors, who take full responsibility for the correctness and
content of the manuscript.
\end{AIuse}

\section{Centre reconstruction and Pansu ball recognition}
\label{sec:weak-characterization}

\subsection{Conventions and centre reconstruction}
\label{sec:centre-reconstruction}

Write $x=(z,t)$ for a point and $\Ca=(z_0,t_0)$ for its
reconstructed centre.

Let $T=\partial_t$ and use
\[
 X_k=\partial_{x_k}+y_k\partial_t,\qquad
 Y_k=\partial_{y_k}-x_k\partial_t,\qquad [X_k,Y_k]=-2T,\qquad JX_k=Y_k.
\]
The group law is
$(z,t)\cdot(w,s)=(z+w,t+s-\langle Jz,w\rangle)$.

The contact and horizontal volume forms are
\[
 \Theta=dt-\sum_k(y_k\dd x_k-x_k\dd y_k),\qquad
 \Omega_H=dx_1\wedge dy_1\wedge\cdots\wedge dx_n\wedge dy_n.
\]
Put $\xi=\ker\Theta$. Thus $d\Theta|_\xi=2\sum_kdx_k\wedge dy_k$, and
$\Theta\wedge\Omega_H=dt\wedge\Omega_H$ is Haar volume.
The Levi metric on $\xi$ is $\LTheta=\tfrac12d\Theta(\cdot,J\cdot)$;
put $\gOne=\LTheta+\Theta^2$.
The frame $X_1,Y_1,\ldots,X_n,Y_n,T$ is orthonormal for $\gOne$.
Horizontal vectors are identified with their coefficients in $(X_k,Y_k)$.

For $V=\sum_k(f_kX_k+g_kY_k)$, put
$\divh V=\sum_k(X_kf_k+Y_kg_k)$.
For an open set $U\subset\Hn$ and $f\in L^1_{\mathrm{loc}}(U)$,
define the distributional horizontal derivative by
\[
 \langle D_Hf,V\rangle
 =-\int_U f\,\divh V\dd\mathcal L^{2n+1}
 \qquad (V\in C_c^\infty(U;\xi)),
\]
where $\mathcal L^{2n+1}$ is Haar measure in these coordinates.
Let $\mathcal M(U;\RR^{2n})$ denote the finite $\RR^{2n}$-valued
Radon measures on $U$, and set
\[
 BV_H(U):=\bigl\{f\in L^1(U):D_Hf\in\mathcal M(U;\RR^{2n})\bigr\}.
\]
For $f\in BV_H(U)$, let $|D_Hf|$ be the total variation measure of
$D_Hf$ (see \cite{FSSC}). Its mass on $U$ is
\[
 |D_Hf|(U)=\sup_{\substack{V\in C_c^\infty(U;\xi)\\ |V|\le1}}
       \left(-\int_U f\,\divh V\dd\mathcal L^{2n+1}\right).
\]
For $f_j,f\in BV_H(U)$, \emph{strict convergence in $BV_H(U)$} means
\[
 f_j\to f\quad\hbox{in }L^1(U),\qquad
 |D_Hf_j|(U)\to|D_Hf|(U).
\]

A measurable set $E$, possibly of infinite volume, has finite horizontal
perimeter if $D_H\mathbf1_E$ is a finite vector measure. Put
\[
 P_H(E)=|D_H\mathbf1_E|(\Hn),\qquad
 \mu_E=|D_H\mathbf1_E|,\qquad
 D_H\mathbf1_E=-\nu_E\mu_E.
\]
The polar decomposition determines the outward unit horizontal normal
$\nu_E$ $\mu_E$-almost everywhere. The minus sign records the drop
of $\mathbf1_E$ from $1$ to $0$ across an outward boundary crossing.

On a smooth hypersurface $\Sigma$, let
$\sigma_\Sigma=\Haus_{\gOne}^{2n}|_\Sigma$ be its Riemannian surface measure.
For the outward $\gOne$-unit normal $N=N_b+N_TT$, a point is
\emph{characteristic} when $N_b=0$ and noncharacteristic otherwise.
Write $\Sigma^*=\{x\in\Sigma:N_b(x)\ne0\}$. On the whole smooth patch,
\begin{equation}\label{sel:tag-2-1}
 d\sigma_H=|N_b|\,d\sigma_\Sigma.
\end{equation}
Write $\Omega_a$ for the unit-radius Pansu ball centred at $a$ and
$\PSphere_a=\partial\Omega_a$. The centre reconstruction is
\cite[Lemma~2.3]{Rigidity}
\begin{equation}\label{eq:centres}
 \begin{aligned}
 z_0&=z-|N_b|N_b-N_TJN_b,\\
 t_0&=t+\langle Jz_0,z\rangle
       -\tfrac12\bigl(\arcsin N_T+N_T|N_b|\bigr),\qquad
 \Ca=(z_0,t_0).
 \end{aligned}
\end{equation}
Since $\langle N_b,JN_b\rangle=0$ and $|N_b|^2+N_T^2=1$,
\begin{equation}\label{eq:profile}
 \begin{aligned}
 |z-z_0|&=|N_b|,\\
 2|t+\langle Jz_0,z\rangle-t_0|
 &=\arccos|N_b|+|N_b|\sqrt{1-|N_b|^2}.
 \end{aligned}
\end{equation}
Comparison with \eqref{eq:sharp-constant} gives $x\in\PSphere_{\Ca(x)}$,
including at characteristic points.
The map $v\mapsto|v|v$ is $C^1$, and
$s\mapsto\tfrac12(\arcsin s+s\sqrt{1-s^2})$ has derivative
$\sqrt{1-s^2}$ and a $C^1$ constant continuation beyond $[-1,1]$.
Thus \eqref{eq:centres} is $C^1$ in the position and unit normal,
and $\Ca$ is $C^1$ on every oriented $C^2$ patch.
The same formulas give
\[
 |z_0-z|\le1,\qquad
 |t_0-t|\le|z_0|\,|z|+\pi/4.
\]

For a function on $\Sigma^*$, $D_H^\Sigma$ denotes differentiation along
$\xi\cap T\Sigma^*$. For vector-valued functions its norm is
Hilbert--Schmidt, with $\gOne$ on the domain and the Euclidean
metric on the target.
In integrals against $\sigma_H$, extend these derivatives by zero
at characteristic points.

\subsection{Pansu ball recognition criterion}

The reconstruction \eqref{eq:centres} uses unit-radius spheres;
Section~\ref{sel:sec:selection} gives the required normalization.

\begin{samepage}
\begin{theorem}[Pansu-ball recognition criterion]
\ifcsname r@thm:weak-limit-characterization\endcsname%
\bookmark[level=3,dest={\getrefbykeydefault{thm:weak-limit-characterization}{anchor}{}}]%
  {Theorem \ref*{thm:weak-limit-characterization}: Pansu-ball recognition criterion}%
\fi
\label{thm:weak-limit-characterization}
Let $n\ge2$ and let $E$ be a bounded set of finite horizontal
perimeter with $0<|E|<\infty$ that globally minimizes horizontal
perimeter among sets of volume $|E|$.
Suppose there are sets
$F_j$ of finite Riemannian perimeter satisfying the following three
conditions.
\begin{enumerate}[label=\textup{(B\arabic*)},leftmargin=3em]
\item \textbf{Strict $BV_H$ convergence.}
\[
 \mathbf1_{F_j}\to\mathbf1_E
 \quad\hbox{strictly in }BV_H(\Hn).
\]
\item \textbf{Regular boundary parts and vanishing cutoff variation.}
The horizontal perimeter measures
$\sigma_{H,j}=|D_H\mathbf1_{F_j}|$ are carried by oriented regular
parts $\Sigma_j\subset\partial F_j$ of class $C^2$.
Use the outward normals and set
\[
 \sigma_j=\Haus_{\gOne}^{2n}|_{\Sigma_j},\qquad
 d\sigma_{H,j}=|N_{b,j}|\,d\sigma_j.
\]
For each fixed $j$, there are cutoffs
$\eta_{j,\ell}\in C_c^1(\Sigma_j)$, $0\le\eta_{j,\ell}\le1$,
such that, as $\ell\to\infty$,
\[
 \eta_{j,\ell}\to1\quad \sigma_{H,j}\hbox{-almost everywhere},
 \qquad
 \int_{\Sigma_j}|D_H^{\Sigma_j}\eta_{j,\ell}|\dd\sigma_{H,j}\to0.
\]
\item \textbf{Vanishing total centre variation.}
For the reconstruction \eqref{eq:centres} at unit scale,
\[
 \int_{\Sigma_j}|D_H^{\Sigma_j}\Ca_j|\dd \sigma_{H,j}\to0
 \qquad\hbox{as }j\to\infty.
\]
\end{enumerate}
Then there is $a\in\Hn$ such that $E=\Omega_a$ almost everywhere.
\end{theorem}
\end{samepage}

In (B2), $\sigma_j$ is Riemannian area and $\sigma_{H,j}$ is
horizontal perimeter; ``carried by $\Sigma_j$'' means
$\sigma_{H,j}(\Hn\setminus\Sigma_j)=0$.
The cutoff limit $\ell\to\infty$ precedes $j\to\infty$.
All centre norms are Euclidean coordinate norms.

Condition (B3) controls variation on each component. Identifying
their centres requires constancy of closed subcurrents,
proved in Section~\ref{sec:one-centre-recognition}.

\section{Pinned elliptic approximations}
\label{sec:approximation}

Global minimizers exist at every positive volume, and each has a
bounded open representative~\cite{LR}. Fix any such representative
$E_0$.

\subsection{Normalization and pinning}
\label{sel:sec:selection}
The dilations $\delta_r(z,t)=(rz,r^2t)$, $r>0$, scale volume by
$r^Q$ and perimeter by $r^{Q-1}$. The unit-radius reconstruction
\eqref{eq:centres} requires the corresponding mean-curvature value
$2n$, so normalize the volume-profile slope of a global minimizer
$E_0$ by choosing
\begin{equation}\label{eq:target-scale}
 \rho_0:=\frac{Q-1}{2nQ}\frac{P_H(E_0)}{|E_0|}>0,
 \qquad E:=\delta_{\rho_0}E_0,\qquad
 \frac{Q-1}{Q}\frac{P_H(E)}{|E|}=2n.
\end{equation}
The scaling laws give the last identity and preserve the
isoperimetric ratio and global minimality.

For a set $F$ of finite $\gOne$-Riemannian perimeter $P_1(F)$,
let $\partial^*F$ denote its reduced boundary: at its points, $F$
has a measure-theoretic tangent half-space and an outward
$\gOne$-unit normal $N_F$. The Riemannian perimeter measure is
concentrated on $\partial^*F$. Write $N_F=N_{b,F}+N_{T,F}T$ and put
\[
 \Phi_\eps(N)=\sqrt{|N_b|^2+\eps^2N_T^2},
 \qquad
 P_\eps(F)=\int_{\partial^*F}\Phi_\eps(N_F)\dd\Haus_{\gOne}^{2n}.
\]
Since $|N_{b,F}|^2+N_{T,F}^2=1$, $\Phi_1(N_F)=1$.
For $\eps>0$, set $P_\eps(F)=+\infty$ when $P_1(F)=+\infty$.
We write $P_0=P_H$, with the surface measure in
\eqref{sel:tag-2-1}.

The parameter $\eps>0$ gives the vertical normal component a
positive cost; $\tau_j$ below sets the pinning strength.
All variations are taken at fixed $\eps_j>0$.

\begin{proposition}
\ifcsname r@prop:normalized-approximants\endcsname%
\bookmark[level=3,dest={\getrefbykeydefault{prop:normalized-approximants}{anchor}{}}]%
  {Proposition \ref*{prop:normalized-approximants}: Normalized approximants}%
\fi
\label{prop:normalized-approximants}
For every bounded open set $E$ that globally minimizes horizontal
perimeter among sets of volume $|E|$ and is normalized by
\eqref{eq:target-scale}, there are finite
Riemannian-perimeter sets $F_j$ in a common compact set satisfying \textup{(B1)--(B3)} of
Theorem~\ref{thm:weak-limit-characterization}.
\end{proposition}

The selection in this section gives (B1). Regularity and cutoffs in
Section~\ref{sel:sec:capacity} give (B2), and the quantitative centre
estimate in Subsection~\ref{sec:centre-energy} gives (B3).

On a smooth boundary patch $\Sigma\subset\partial F$, a compactly
supported ambient field $W$ with boundary trace supported in $\Sigma$
has volume derivative
\[
 \delta V[W]=\int_\Sigma\langle W,N\rangle\dd\sigma_\Sigma.
\]

On the noncharacteristic regular part of $\partial F_j$, use the
subscript $j$ for these normal data, and write
$\sigma_{H,j}=|D_H\mathbf1_{F_j}|$ for horizontal perimeter.
First variations are integrated against the full reduced-boundary
measure $\Haus_{\gOne}^{2n}|_{\partial^*F_j}$.
For ambient functions, $C^2(U)$ uses the fixed coordinates $(z,t)$.

Choose a fixed $\psi_E\in C_c^\infty(\Hn)$ with
$0\le\psi_E\le1$ and $\{\psi_E>0\}=E$. To construct it, take
nonnegative smooth bumps $\beta_\ell$ supported in $E$ whose
positive sets cover $E$, and sum
$2^{-\ell}\beta_\ell/(1+\|\beta_\ell\|_{C^\ell})$.
The series converges in every $C^k$ norm on the ambient space;
its support is contained in the compact set $\overline E$.

Define the pinning functional
\begin{equation}\label{sel:eq:local-pinning}
 \mathscr P_E(F):=\int_{E\setminus F}\psi_E\dd x
 =\int_E\psi_E\dd x-\int_F\psi_E\dd x\ge0.
\end{equation}
It satisfies
$|\mathscr P_E(F)-\mathscr P_E(G)|\le|F\triangle G|$.

For any set $A$ of finite volume, with $v_0=|E|$,
$|A\triangle E|\le2|E\setminus A|+||A|-v_0|$.
For every $\delta>0$, split $E\setminus A$ at the level $\psi_E=\delta$:
\[
 \begin{aligned}
 |E\setminus A|
 &=|(E\setminus A)\cap\{\psi_E\le\delta\}|
   +|(E\setminus A)\cap\{\psi_E>\delta\}|\\
 &\le|E\cap\{\psi_E\le\delta\}|
       +\delta^{-1}\int_{E\setminus A}\psi_E\dd x.
 \end{aligned}
\]
Consequently,
\begin{equation}\label{sel:eq:indicator-pinning}
 |A\triangle E|
 \le2|E\cap\{\psi_E\le\delta\}|
       +2\delta^{-1}\mathscr P_E(A)+||A|-v_0|.
\end{equation}
If $|A_j|\to v_0$ and $\mathscr P_E(A_j)\to0$, then
$\mathbf1_{A_j}\to\mathbf1_E$ globally in $L^1$:
first let $j\to\infty$ at fixed $\delta$, then $\delta\downarrow0$,
using $\psi_E>0$ on $E$.

\subsection{Selection, confinement, and stability}
\label{sec:volume-repair-selection}

Set $v_0=|E|$. Dilation and the global
minimality of $E$ determine the isoperimetric profile at every volume:
\begin{equation}\label{sel:tag-2-11}
 I(v):=\inf_{|F|=v}P_H(F)
 =P_H(E)(v/v_0)^{(Q-1)/Q}\quad(v>0),\qquad I(0)=0.
\end{equation}
Indeed, scaling a competitor of volume $v$ to volume $v_0$ gives
the lower bound, and a dilation of $E$ attains it.
The coefficient in \eqref{sel:tag-2-11} is determined by $E$;
its identification with $C_n$ follows from the classification.
In particular, $I'(v_0)=2n$.

On a smooth oriented patch $\Sigma\subset\partial F$, define the
outward $\eps$-mean curvature by
\[
 \delta P_\eps(F)[X]=\int_\Sigma H_\eps\langle X,N\rangle
                       \dd\sigma_\Sigma,
\]
for ambient fields $X$ whose support meets $\partial F$ in a compact
subset of $\Sigma$.

The term $P_{\eps_j}(F)-I(|F|)$ controls the perimeter deficit.
The quadratic volume term selects $v_0$ without imposing $|F|=v_0$.
The pinning term selects the prescribed minimizer $E$ among sets
whose volumes approach $v_0$.

\begin{lemma}
\ifcsname r@lem:selection-confinement\endcsname%
\bookmark[level=3,dest={\getrefbykeydefault{lem:selection-confinement}{anchor}{}}]%
  {Lemma \ref*{lem:selection-confinement}: Selection and confinement}%
\fi
\label{lem:selection-confinement}
Let $E$ be as in Proposition~\ref{prop:normalized-approximants},
put $v_0=|E|$, and fix $\tau_j=2^{-j}$.
There is a bounded open Euclidean ball $\mathcal D$ such that,
for all sufficiently large $j$, one can choose
$0<\eps_j\le\tau_j$ decreasing to zero and minimizers
\[
 F_j\in\operatorname*{argmin}_{F\subset\mathcal D}\mathfrak G_j(F),
 \qquad
 \mathfrak G_j(F):=
 P_{\eps_j}(F)-I(|F|)
 +(|F|-v_0)^2+\tau_j\mathscr P_E(F).
\]
$P_{\eps_j}(F)$ is computed on all of $\Hn$, including any part of
the reduced boundary of $F$ on $\partial\mathcal D$. Put
$v_j=|F_j|$. Then
\[
 |v_j-v_0|\le\sqrt2\,\tau_j,\qquad
 P_{\eps_j}(F_j)-P_H(F_j)\le2\tau_j^2,\qquad
 \mathscr P_E(F_j)\le2\tau_j.
\]
Moreover, $\mathbf1_{F_j}\to\mathbf1_E$ strictly in $BV_H(\Hn)$.
The $F_j$ have representatives in a common closed Euclidean ball
$\mathbb B\Subset\mathcal D$ containing the supports of their Riemannian
perimeter measures.

Define the prescribed mean-curvature function
\begin{equation}\label{eq:prescribed-curvature}
 H_j:=I'(v_j)-2(v_j-v_0)+\tau_j\psi_E.
\end{equation}
On a fixed open set $\mathbb B\Subset U\Subset\mathcal D$,
$\|H_j-2n\|_{C^2(U)}\le C_E\tau_j\to0$.
The functional
$F\mapsto P_{\eps_j}(F)-\int_FH_j\dd x$ is stationary at $F_j$
for compactly supported smooth ambient variations. Its Hessian is
nonnegative for $\gOne$-normal speeds
$v=\langle W,N_j\rangle\in C_c^2(\Sigma_j^*)$ satisfying
$\int_{\Sigma_j}v\dd\sigma_j=0$, where $\Sigma_j^*$ is the
noncharacteristic regular boundary; the integral condition is the
linearized volume constraint. On the regular part,
$H_{\eps_j}=H_j$.
\end{lemma}

\begin{proof}
Smooth recovery sets determine $\eps_j$ so that elliptic regularization
costs at most $\tau_j^2$. For a nonnegative smooth approximate
identity $\rho_\delta$ of
integral one and shrinking compact support, use left group convolution
\begin{equation}\label{eq:left-convolution}
 (\rho_\delta*f)(x)=\int_{\Hn}\rho_\delta(y)f(y^{-1}\cdot x)\dd y.
\end{equation}
For every left-invariant horizontal field $Z$,
$Z(\rho_\delta*f)=\rho_\delta*(Zf)$ in distributions: the curve
$x\exp(sZ)$ becomes $y^{-1}x\exp(sZ)$ under left translation.
For a vector-valued finite measure $\mu$, positivity of the kernel
and invariance of Haar measure give
\begin{equation}\label{eq:convolution-contraction}
 \int_{\Hn}|\rho_\delta*\mu|\dd x\le|\mu|(\Hn).
\end{equation}
Take $u_\delta=\rho_\delta*\mathbf1_E$, so $0\le u_\delta\le1$.
Approximation in $L^1$, the contraction above and lower
semicontinuity give $u_\delta\to\mathbf1_E$ in $L^1$
and $\int_{\Hn}|\nabla_Hu_\delta|\dd x\to P_H(E)$
\cite[Theorems~2.4 and~3.1]{AGM2015}.
Euclidean coarea~\cite[Chapter~13]{Maggi} with weight
$|\nabla_Hu_\delta|/|\nabla u_\delta|$ gives
\[
 \int_0^1P_H(\{u_\delta>s\})\dd s
 =\int_{\Hn}|\nabla_Hu_\delta|\dd x.
\]
Together with $L^1$ convergence and lower semicontinuity,
this selects regular levels with smooth bounded sets $G_k$ and
$\mathbf1_{G_k}\to\mathbf1_E$ strictly in $BV_H(\Hn)$.
Mollifiers with shrinking compact support confine all positive
level sets to a fixed compact neighbourhood of $\overline E$.
Choose a bounded open Euclidean ball $\mathcal D$ containing this
neighbourhood, and a closed concentric ball $\mathbb B\Subset\mathcal D$
with $\overline E\subset\operatorname{int}\mathbb B$.
Since $|G_k|\to v_0$,
\[
 P_H(G_k)-I(|G_k|)\to0,\qquad
 (|G_k|-v_0)^2\to0,\qquad
 \mathscr P_E(G_k)\to0.
\]
Choose a subsequence, denoted $\widetilde E_j$, so that
the sum of these three nonnegative terms, with the last weighted
by $\tau_j$, is at most $\tau_j^2$. Each recovery set has finite
Riemannian perimeter. Recursively choose
$0<\eps_j\le\min\{\tau_j,\eps_{j-1}/2\}$, with $\eps_0=1$,
so that $\eps_jP_1(\widetilde E_j)\le\tau_j^2$.
Since $0\le P_{\eps_j}-P_H\le\eps_jP_1$,
\begin{equation}\label{sel:tag-2-13}
 \mathfrak G_j(\widetilde E_j)\le2\tau_j^2.
\end{equation}

The same comparison selects minimizers converging strictly to $E$.
For every $F\subset\mathcal D$ of finite Riemannian perimeter,
\[
 \mathfrak G_j(F)=
 \bigl(P_H(F)-I(|F|)\bigr)
 +\bigl(P_{\eps_j}(F)-P_H(F)\bigr)
 +(|F|-v_0)^2
 +\tau_j\mathscr P_E(F)\ge0.
\]
Since $0<\eps_j\le1$ and $F\subset\mathcal D$,
\[
 \eps_jP_1(F)\le P_{\eps_j}(F)
 \le\mathfrak G_j(F)+I(|\mathcal D|).
\]
Thus a minimizing sequence has bounded $P_1$ at fixed $j$.
Extend the sets by zero to a larger ball. $BV$ compactness
\cite[Section~12.4]{Maggi} gives
an $L^1$ limit contained in $\mathcal D$ up to a null set.
Lower semicontinuity of $P_{\eps_j}$ and continuity of the volume,
$I$, and pinning terms show that the limit is a minimizer $F_j$.
Comparison with \eqref{sel:tag-2-13} yields
\begin{equation}\label{sel:tag-2-28}
 \bigl(P_H(F_j)-I(v_j)\bigr)+\bigl(P_{\eps_j}(F_j)-P_H(F_j)\bigr)
 +(v_j-v_0)^2
 +\tau_j\mathscr P_E(F_j)
 \le2\tau_j^2.
\end{equation}
Thus $|v_j-v_0|\le\sqrt2\,\tau_j$,
$P_{\eps_j}(F_j)-P_H(F_j)\le2\tau_j^2$, and $\mathscr P_E(F_j)\le2\tau_j$.
Equation~\eqref{sel:eq:indicator-pinning} gives global $L^1$
convergence to $E$. The perimeter bounds are
\[
 I(v_j)\le P_H(F_j)\le P_{\eps_j}(F_j)
 \le I(v_j)+2\tau_j^2\longrightarrow I(v_0)=P_H(E).
\]

To confine their boundary supports, deletion competitors give a
local volume lower bound independent of $\eps_j$.
For all sufficiently large $j$, the volume bound gives $v_j\ge v_0/2$.
Since $0<(Q-1)/Q<1$, one has
$(w/v_j)^{(Q-1)/Q}\ge w/v_j$ for $0\le w\le v_j$; for
$w\ge v_j$, $I(v_j)-I(w)\le0$. Hence
\[
 I(v_j)-I(w)\le\frac{I(v_j)}{v_j}(v_j-w)_+
 \le C|v_j-w|\qquad(w\ge0),
\]
with $C$ independent of $j$. For $w\in[0,|\mathcal D|]$,
the quadratic difference has absolute value
$|v_j-w|\,|v_j+w-2v_0|\le C|v_j-w|$,
and $\mathscr P_E$ is $1$-Lipschitz in symmetric difference.
Minimizing $\mathfrak G_j$ over $F\subset\mathcal D$ therefore gives,
for every $G\subset\mathcal D$,
\begin{equation}\label{sel:eq:global-set-comparison}
 P_{\eps_j}(F_j)\le P_{\eps_j}(G)+C|F_j\triangle G|.
\end{equation}

Write $d_{\gOne}$ for the distance of $\gOne$ and
$B^{\gOne}_r(x)$ for its metric balls. Since
$\overline E\subset\operatorname{int}\mathbb B$,
choose $0<r_*<1$ with
$r_*<\tfrac12 d_{\gOne}(\overline E,\partial\mathbb B)$.
Since $F_j\setminus B_r^{\gOne}(x)\subset\mathcal D$,
\eqref{sel:eq:global-set-comparison} applies to this deletion competitor.

Let $D$ denote the Euclidean distributional derivative.
Fix $x\in\spt|D\mathbf1_{F_j}|$ and put
\[
 \begin{aligned}
 m(r)&=|F_j\cap B^{\gOne}_r(x)|,\\
 h_j(r)&=P_{\eps_j}(F_j\cap B^{\gOne}_r(x);\partial B^{\gOne}_r(x)).
 \end{aligned}
\]
For almost every $r<r_*$, $h_j(r)$ is the perimeter created by
the cut. Since $\Phi_{\eps_j}\le1$ on unit normals, coarea
\cite[Chapter~13]{Maggi} gives
$h_j(r)\le m'(r)$.
The deletion comparison \eqref{sel:eq:global-set-comparison} and
the profile bound \eqref{sel:tag-2-11} for \(F_j\cap B^{\gOne}_r(x)\)
give
\[
 \begin{split}
 P_{\eps_j}(F_j;B^{\gOne}_r(x))
 &\le h_j(r)+Cm(r),\\
 I(m(r))\le P_{\eps_j}(F_j\cap B^{\gOne}_r(x))
 &\le2h_j(r)+Cm(r).
 \end{split}
\]
Since $I(v)$ is a positive constant times $v^{(Q-1)/Q}$,
\[
 m(r)^{(Q-1)/Q}\le C\{m'(r)+m(r)\}.
\]
Boundary support gives $0<m(r)\le|B^{\gOne}_r|$.
Decrease $r_*$ so that $C|B^{\gOne}_{r_*}|^{1/Q}\le1/2$;
left invariance makes the choice independent of $x,j$.
Then $Cm(r)\le\tfrac12m(r)^{(Q-1)/Q}$, so
\[
 m'(r)\ge\frac1{2C}m(r)^{(Q-1)/Q},\qquad
 (m(r)^{1/Q})'=\frac1Qm(r)^{-(Q-1)/Q}m'(r)
       \ge\frac1{2CQ}
\]
almost everywhere. On each interval $[s,r_*]$ with $0<s<r_*$,
$m$ is absolutely continuous and bounded away from zero, so
integration gives $m(r_*)^{1/Q}-m(s)^{1/Q}\ge(r_*-s)/(2CQ)$.
Letting $s\downarrow0$, using $m(s)\le|B^{\gOne}_s|\to0$,
gives, with $c_*=(2CQ)^{-Q}$,
\begin{equation}
|F_j\cap B^{\gOne}_{r_*}(x)|\ge c_*r_*^Q
 \quad(x\in\spt|D\mathbf1_{F_j}|).
\label{sel:tag-2-32}
\end{equation}

If a boundary point $x$ has
$d_{\gOne}(x,\overline E)>r_*$, its radius-$r_*$ ball misses $E$, so
\eqref{sel:tag-2-32} gives
\[
 |F_j\triangle E|\ge|F_j\cap B^{\gOne}_{r_*}(x)|\ge c_*r_*^Q.
\]
Global $L^1$ convergence therefore gives, for all sufficiently large $j$,
\begin{equation}\label{sel:tag-2-33}
 \spt|D\mathbf1_{F_j}|
 \subset\{x:d_{\gOne}(x,\overline E)\le r_*\}\subset\mathbb B\Subset\mathcal D.
\end{equation}
The second inclusion follows from the choice of $r_*$.
Because $\mathbb B$ is a Euclidean ball, $\Hn\setminus\mathbb B$ is connected.
The distributional derivative of $\mathbf1_{F_j}$ vanishes there,
so $\mathbf1_{F_j}$ is constant; finite volume makes that constant
zero. Thus representatives of all $F_j$ lie in $\mathbb B$.

For $F_j\triangle G\Subset B^{\gOne}_r(x)\Subset\mathcal D$, $r<r_*$,
the competitor remains in $\mathcal D$ and locality gives
\begin{equation}
P_{\eps_j}(F_j;B^{\gOne}_r(x))
 \le P_{\eps_j}(G;B^{\gOne}_r(x))+C|F_j\triangle G|.
\label{sel:tag-2-30}
\end{equation}
Confinement makes every compactly supported smooth ambient flow
admissible for sufficiently small time, so the first variation
determines $H_j$.
Under such a flow, the transformed reduced-boundary normal and surface
Jacobian depend smoothly on time, with bounded derivatives for small
times. At fixed $j$, smoothness of $\Phi_{\eps_j}$ on unit normals
and $P_1(F_j)<\infty$ justify differentiation under the perimeter integral.

By \eqref{sel:tag-2-28}, smoothness of $I'$ near $v_0$, and
\eqref{eq:prescribed-curvature},
$\|H_j-2n\|_{C^2(U)}\le C_E\tau_j$.
Differentiating $\mathfrak G_j$ at $F_j$ gives
\begin{equation}\label{sel:tag-2-36}
 \delta P_{\eps_j}(F_j)[X]
 -\int_{\partial^*F_j}H_j
       \langle X,N_j\rangle\dd\Haus_{\gOne}^{2n}=0.
\end{equation}
On a regular boundary patch, \eqref{sel:tag-2-36} implies
$H_{\eps_j}=H_j$.

The second variation gives stability under the linearized volume
constraint. For a $C_c^2$ $\gOne$-normal variation on $\Sigma_j^*$,
keep $H_j$ fixed. Its constant part $I'(v_j)-2(v_j-v_0)$ cancels the
terms containing $\delta^2V$, giving
\[
 \delta^2\mathfrak G_j
 =\delta^2\left(P_{\eps_j}-\int_FH_j\dd x\right)
   +(2-I''(v_j))(\delta V)^2.
\]
If $\delta V=0$, minimality of $\mathfrak G_j$ makes the
prescribed-curvature Hessian nonnegative.
Both second variations exist: $\Phi_{\eps_j}$ is smooth on unit
normals at fixed $j$, and $H_j$ is smooth.
\end{proof}

\section{Regularity, cutoffs, and vanishing centre variation}
\label{sel:sec:capacity}

Let $F_j$ be the sets selected in Section~\ref{sec:approximation}.
Retain the fixed open set $\mathbb B\Subset U\Subset\mathcal D$.
Choose representatives with $\partial F_j=\spt|D\mathbf1_{F_j}|$.
Let $\Sigma_j$ consist of points where $\partial F_j$ is locally
a $C^1$ graph, put $S_j=\partial F_j\setminus\Sigma_j$, and set
$m=2n$ and $\sigma_j=\Haus_{\gOne}^{m}|_{\Sigma_j}$.
Characteristic points in $\Sigma_j$ satisfy $N_{b,j}=0$.
Regularity and cutoff limits are taken at fixed $j$; their constants
may depend on $j$.

\Needspace{7\baselineskip}
\subsection{Boundary regularity and local area bounds}

\begin{proposition}
\ifcsname r@prop:selected-regularity\endcsname%
\bookmark[level=3,dest={\getrefbykeydefault{prop:selected-regularity}{anchor}{}}]%
  {Proposition \ref*{prop:selected-regularity}: Regular boundary and local area bound}%
\fi
\label{prop:selected-regularity}
For each fixed $j$, $\Sigma_j$ is a smooth outward-oriented
hypersurface carrying the full Riemannian perimeter of $F_j$.
The singular set $S_j$ is compact, is empty when
$m+1\le7$, and otherwise has Hausdorff dimension at most $m-7$.
In particular,
\begin{equation}\label{sel:tag-3-2}
 \Haus^{m-2}(S_j)=0.
\end{equation}
There are $c_j,r_j>0$ such that
\begin{equation}\label{sel:eq:area-upper}
 \sigma_j(B_r(x))\le c_jr^m
 \qquad(x\in\partial F_j,\ 0<r<r_j),
\end{equation}
where the balls are Euclidean. Smoothness includes regular
characteristic points.
\end{proposition}

\begin{proof}
Let $\mathsf B$ have the left-invariant frame as columns, so $\det\mathsf B=1$.
For the outward Euclidean unit normal $\nu_F^{\rm e}$ on $\partial^*F$,
the frame coefficients and surface measure satisfy
\[
 (N_{b,F},N_{T,F})=\frac{\mathsf B^\top\nu_F^{\rm e}}{|\mathsf B^\top\nu_F^{\rm e}|},
 \qquad
 d\Haus_{\gOne}^m=|\mathsf B^\top\nu_F^{\rm e}|\,d\Haus^m.
\]
Set $M_j=\mathsf B\operatorname{diag}(I_{2n},\eps_j^2)\mathsf B^\top$.
The normalization factors cancel in $P_{\eps_j}$:
\begin{equation}\label{sel:eq:elliptic-metric}
 P_{\eps_j}(F;U)=\int_{U\cap\partial^*F}
 \sqrt{(\nu_F^{\rm e})^\top M_j(x)\nu_F^{\rm e}}\dd\Haus^m(x).
\end{equation}

The matrix $M_j$ is smooth and positive definite near $\overline U$.
A smooth compactly supported cutoff $0\le\chi\le1$, equal to one there,
gives the uniformly positive definite extension
$\widetilde M_j=\chi M_j+(1-\chi)I$, which equals $I$ outside
a compact set and has bounded first derivatives. In particular
it is $C^{0,1/2}$.

For $F_j\triangle G\Subset U\cap B_r(x)$ and $r\le1$,
\eqref{sel:eq:global-set-comparison} gives an additive error
at most $c_j|F_j\triangle G|\le c_jr^{m+1}\le c_jr^{m+1/2}$.
The extension agrees with $M_j$ on the modification region, so
Simmons's almost-minimality condition holds in ambient dimension
$m+1$, with exponent $1/2$.
By \cite[Theorem~1.1]{Simmons2022}, the reduced boundary is relatively
open in $\partial F_j$ and is $C^{1,1/8}$, and the stated singular-set
bounds follow. The singular set is closed in the compact boundary
support. Lemma~4.4 of the same source gives the Euclidean perimeter
bound $c_jr^m$. Compactness of the support inside $U$ permits one
radius $r_j$; comparison of Euclidean and $\gOne$ surface measure
on $\overline U$ gives \eqref{sel:eq:area-upper}.

To improve a regular patch, choose orthonormal affine Euclidean
coordinates $(y,s)$ in which the surface is $s=w(y)$,
$y\in V_g\subset\mathbb R^m$, with the set below the graph.
Here $w\in C^{1,\beta}$, $\beta=1/8$, and Haar volume is $dy\,ds$.
Shrink the chart so that $|Dw|\le L$.
In these coordinates the graph density is
\begin{equation}\label{sel:eq:graph-density}
 \Psi_j(y,s,q)=\sqrt{\zeta(q)^\top M_j(y,s)\zeta(q)},
 \qquad \zeta(q)=(-q,1).
\end{equation}
For $\widehat h=(-h,0)$, $M=M_j(y,s)$ and $\Psi=\Psi_j(y,s,q)$,
\[
 D^2_{qq}\Psi[h,h]
 =\frac{|M^{1/2}\widehat h\wedge M^{1/2}\zeta|^2}{\Psi^3}>0
 \qquad(h\ne0).
\]
The two vectors are independent, and $M$ is positive definite.
Compactness of the graph subcylinder and $|q|\le L$ give constants
$0<\lambda_j\le\Lambda_j<\infty$ such that
\begin{equation}\label{sel:eq:graph-ellipticity}
 \lambda_j|h|^2\le D^2_{qq}\Psi_j[h,h]\le\Lambda_j|h|^2.
\end{equation}
These bounds include characteristic points, with $\eps_j>0$ fixed.

A compactly supported graph-height variation $\varphi$ has volume
derivative $\int_{V_g}\varphi\,dy$. Thus \eqref{sel:tag-2-36} gives,
in distributions,
\begin{equation}\label{sel:eq:graph-euler}
 \operatorname{div}_y\Psi_{j,q}(y,w,Dw)=b,
 \qquad b(y)=\Psi_{j,s}(y,w,Dw)-H_j(y,w(y)).
\end{equation}

For coordinate difference quotients
$v_h(y)=[w(y+he_i)-w(y)]/h$, which belong to $H^1$ on smaller
charts when $h\ne0$, \eqref{sel:eq:graph-euler} gives a linear
divergence-form equation with zero scalar forcing: indeed,
$\delta_hb=\partial_i\int_0^1b(y+the_i)\,dt$.
Since $v_h$ is an average of $\partial_iw$, the mean-value formula,
smoothness of $\Psi_j,H_j$, and \eqref{sel:eq:graph-ellipticity}
give ellipticity, $C^{0,\beta}$ bounds for the coefficients and
vector forcing, and an $L^2$ bound for $v_h$, uniformly in $h$
at fixed $j$.
With $p>m/(1-\beta)$, \cite[Theorem~6.2]{Vita2026} gives a uniform
local $C^{1,\beta}$ bound for $v_h$, hence $w\in C^{2,\beta}_{\rm loc}$.
Differentiating \eqref{sel:eq:graph-euler} and iterating
\cite[Theorem~6.3]{Vita2026} gives smoothness, including at
characteristic points. All regularity constants may depend on $j$.
Thus $\Sigma_j=\partial^*F_j$ carries all Riemannian perimeter.
\end{proof}

\needspace{7\baselineskip}
\subsection{Cutoffs avoiding singular and characteristic sets}

\begin{proposition}
\ifcsname r@prop:joint-cutoffs\endcsname%
\bookmark[level=3,dest={\getrefbykeydefault{prop:joint-cutoffs}{anchor}{}}]%
  {Proposition \ref*{prop:joint-cutoffs}: Cutoffs avoiding both exceptional sets}%
\fi
\label{prop:joint-cutoffs}
For each fixed $j$ there are $\eta_{j,\ell}\in C_c^1(\Sigma_j^*)$,
$0\le\eta_{j,\ell}\le1$, converging locally uniformly to one on
$\Sigma_j^*$ and $\sigma_j$-almost everywhere on $\Sigma_j$, with
\begin{equation}\label{sel:tag-3-2b}
 \int_{\Sigma_j}|\nabla_{\Sigma_j}\eta_{j,\ell}|^2\dd\sigma_j
 \longrightarrow0.
\end{equation}
The limit is taken as $\ell\to\infty$ with $j$ fixed.
\end{proposition}
\begin{proof}
If $S_j$ is empty, $\Sigma_j$ is compact and we take
$\widetilde\eta_{j,\ell}=1$. Otherwise, apply the cutoff construction
of \cite[Section~2.1 and Lemma~2.1, $q=2$]{Zhu2016}, using
\eqref{sel:tag-3-2} and \eqref{sel:eq:area-upper}.
Its covering argument uses only compactness of $S_j$, its vanishing
$\Haus^{m-2}$ measure, and the local area bound. It therefore applies
in Euclidean balls here; comparison with the smooth metric $\gOne$
on the common compact support changes only the fixed-$j$ constant.
Choose covering radii below $1/\ell$ and with sum of their
$(m-2)$-powers below $1/\ell$. The resulting Lipschitz cutoffs
$0\le\theta_{j,\ell}\le1$ have compact support in $\Sigma_j$,
vanish near $S_j$, converge locally uniformly to one, and satisfy
\[
 \int_{\Sigma_j}|\nabla_{\Sigma_j}\theta_{j,\ell}|^2\dd\sigma_j
 \le c_j/\ell.
\]
Nonnegative local mollification in finitely many regular charts,
combined with a smooth nonnegative partition of unity, gives
$\widetilde\eta_{j,\ell}\in C_c^1(\Sigma_j)$ with values in $[0,1]$,
support away from $S_j$, and uniform and $W^{1,2}(\Sigma_j)$
errors less than $1/\ell$. Thus
\[
 \int_{\Sigma_j}|\nabla_{\Sigma_j}\widetilde\eta_{j,\ell}|^2\dd\sigma_j
 \le2\int_{\Sigma_j}|\nabla_{\Sigma_j}\theta_{j,\ell}|^2\dd\sigma_j
       +2/\ell^2\longrightarrow0.
\]
The uniform approximation preserves local uniform convergence to one.

Near a characteristic point, write the regular surface as
$t=f(z)$ and put $G(z)=\nabla f(z)+Jz$.
Since $DG-DG^\top=2J$,
$2n=\operatorname{rank}(DG-DG^\top)\le2\operatorname{rank}DG$.
For $n\ge2$, choose components $G_a,G_b$ with independent
differentials and use $y=(G_a,G_b)\in\RR^2$ as transverse coordinates.
The characteristic set lies in $\{y=0\}$,
so a countable chart cover makes it $\sigma_j$-null.

Put $g=|N_b|$. Near a characteristic point, $g$ is comparable
to $|G|$, and $|(G_a,G_b)|\le|G|$. The transverse coordinates and a
finite cover of the zeros in each compact $K\subset\Sigma_j$ give
\[
 \sigma_j(K\cap\{g<t\})\le c_{j,K}t^2
\]
for small $t>0$; outside the cover, $g$ has a positive lower bound.

This quadratic sublevel bound is the borderline case for vanishing
Dirichlet energy: a logarithmic transition supplies the required
decay, including when $n=2$.
Fix $\ell$ and put $K_\ell=\spt\widetilde\eta_{j,\ell}$.
For $0<\rho<1/2$, choose a smooth cutoff $\zeta_\rho$ that is
zero on $[0,\rho^2]$, one on $[\rho,\infty)$, and satisfies
$|\zeta_\rho'(t)|\le C/(t|\log\rho|)$ between these levels.
The function $q_\rho=\zeta_\rho(g)$ is $C^1$, including near
$g=0$. The sublevel bound and boundedness of $\nabla_{\Sigma_j}g$
on $K_\ell\cap\{g>0\}$ give
\[
 \begin{aligned}
 \int_{K_\ell}|\nabla_{\Sigma_j}q_\rho|^2\dd\sigma_j
 &\le\frac{c_{j,\ell}}{|\log\rho|^2}
       \int_{K_\ell\cap\{\rho^2<g<\rho\}}g^{-2}\dd\sigma_j\\
 &\le\frac{c_{j,\ell}}{|\log\rho|^2}
       \left(1+\int_{\rho^2}^{\rho}\frac{\dd t}{t}\right)
 \le\frac{c_{j,\ell}}{|\log\rho|}\longrightarrow0.
 \end{aligned}
\]
The second inequality follows by integration by parts against
$t\mapsto\sigma_j(K_\ell\cap\{g<t\})$.
Choose $\rho_{j,\ell}\le1/\ell$ so that this energy is at most
$1/\ell$, and set
$\eta_{j,\ell}=\widetilde\eta_{j,\ell}q_{\rho_{j,\ell}}$.
Its support is compact in $\Sigma_j^*$, and
\[
 \int_{\Sigma_j}|\nabla_{\Sigma_j}\eta_{j,\ell}|^2\dd\sigma_j
 \le2\int_{\Sigma_j}|\nabla_{\Sigma_j}\widetilde\eta_{j,\ell}|^2\dd\sigma_j
       +2/\ell\longrightarrow0.
\]
On every compact subset of $\Sigma_j^*$, $g$ has a positive
lower bound, so $q_{\rho_{j,\ell}}$ eventually equals one.
Local uniform convergence of $\widetilde\eta_{j,\ell}$ and
nullity of the characteristic set give both asserted convergence
statements.

The surface measure identity and $|N_{b,j}|\le1$ give
\[
 \int_{\Sigma_j}|\nabla_{\Sigma_j}\eta_{j,\ell}|\dd\sigma_{H,j}
 \le P_H(F_j)^{1/2}
 \left(\int_{\Sigma_j}|\nabla_{\Sigma_j}\eta_{j,\ell}|^2
                  \dd\sigma_j\right)^{1/2}\longrightarrow0.
\]
Since $|D_H^{\Sigma_j}\eta_{j,\ell}|\le
|\nabla_{\Sigma_j}\eta_{j,\ell}|$, this also gives \textup{(B2)}.
\end{proof}

\subsection{Vanishing horizontal centre variation}
\label{sec:centre-energy}

The centre estimate in \cite[Theorem~4.1]{CentreRigidity} applies
to sets $F_j$ of finite Riemannian perimeter in a fixed compact
set $\mathcal K\Subset U\Subset\Hn$, with positive volume and
horizontal perimeter, $0<\eps_j\le1$, and
$H_j\in C^2(U)$ satisfying $\|H_j-2n\|_{C^2(U)}\le1$.
Their outward-oriented smooth parts $\Sigma_j\subset\partial^*F_j$
must carry all Riemannian perimeter and admit
$\eta_{j,\ell}\in C_c^1(\Sigma_j^*)$, $0\le\eta_{j,\ell}\le1$,
converging to one $\sigma_j$-almost everywhere, with
$\int_{\Sigma_j}|\nabla_{\Sigma_j}\eta_{j,\ell}|^2\dd\sigma_j\to0$
at fixed $j$. If $P_{\eps_j}(F)-\int_FH_j\dd x$ is stationary
under compactly supported smooth ambient variations in $U$ and
has nonnegative Hessian on $\gOne$-normal speeds
$v\in C_c^2(\Sigma_j^*)$ with $\int_{\Sigma_j}v\dd\sigma_j=0$,
then there are $C,\delta_*>0$, depending only on $n,\mathcal K,U$,
such that
\begin{equation}\label{eq:quoted-centre-estimate}
 \begin{gathered}
 \delta_j:=\eps_j^2+\|H_j-2n\|_{C^2(U)}\le\delta_*
 \quad\Longrightarrow\\
 \int_{\Sigma_j}|D_H^{\Sigma_j}\Ca_j|^2\dd\sigma_{H,j}
 \le C\delta_j|F_j|.
 \end{gathered}
\end{equation}
Both $C$ and $\delta_*$ are independent of $\eps_j$ and of the
constants in the fixed-$j$ regularity and cutoff constructions.
The source's $\sigma_{1,j}$ and $\mathsf E_{C,j}$ are
$\sigma_j$ and $D_H^{\Sigma_j}\Ca_j$ here, with the reconstruction
\eqref{eq:centres} at unit scale.

\begin{proof}[Proof of Proposition~\ref{prop:normalized-approximants}]
Lemma~\ref{lem:selection-confinement} gives (B1), confinement in
$\mathcal K=\mathbb B$, and positive limiting volume and perimeter.
Propositions~\ref{prop:selected-regularity} and~\ref{prop:joint-cutoffs}
give (B2) and the smooth full-perimeter part and cutoffs required
by \eqref{eq:quoted-centre-estimate}.
The selection lemma also gives stationarity and Hessian
nonnegativity on zero-volume speeds: the Hessians differ by
$(2-I''(v_j))(\delta V)^2$, which vanishes there.
The identification with the source's closed normal pairing is
\cite[Proposition~2.1 and Appendix~A]{CentreRigidity}.
Finally,
\begin{equation}\label{eq:approximation-error}
 \eps_j^2+\|H_j-2n\|_{C^2(U)}
 \le\tau_j^2+C_E\tau_j\longrightarrow0.
\end{equation}
All hypotheses and the fixed threshold of
\eqref{eq:quoted-centre-estimate} therefore hold for large $j$.
Since $|F_j|$ and $P_H(F_j)$ are bounded, that estimate and
Cauchy--Schwarz give
\begin{equation}\label{eq:full-centre-decay}
 \begin{aligned}
 \int_{\Sigma_j}|D_H^{\Sigma_j}\Ca_j|^2\dd\sigma_{H,j}
 &\le C\tau_j,\\
 \int_{\Sigma_j}|D_H^{\Sigma_j}\Ca_j|\dd\sigma_{H,j}
 &\le P_H(F_j)^{1/2}
    \left(\int_{\Sigma_j}|D_H^{\Sigma_j}\Ca_j|^2\dd\sigma_{H,j}\right)^{1/2}
 \le C\tau_j^{1/2}\longrightarrow0.
 \end{aligned}
\end{equation}
This proves (B3).
\end{proof}

\section{Centre-weighted currents}
\label{sec:weighted-currents}

Throughout this section, $F_j,E,\Sigma_j$ satisfy
\textup{(B1)--(B3)} of Theorem~\ref{thm:weak-limit-characterization}.
In particular, the regular parts are only assumed $C^2$, and the
cutoff assumption is the horizontal $L^1$ condition in (B2).

\subsection{Horizontal boundary currents}
\label{weak:sec:classification}

For a smooth compactly supported horizontal vector field $V$, put
$\omega_V=\iota_V(\Theta\wedge\Omega_H)$.
A current is a continuous linear functional on compactly
supported test forms; $T=\tau\mu$ means that it integrates their
pairing with an orientation $\tau$ against $\mu$.

In the upper degrees of the Rumin complex
\cite{Rumin} and \cite[Sections~3.2 and~3.5]{Vittone2022}, the test forms are
\begin{equation}\label{weak:legacy-eq-2-1}
 J^k=\{\omega\in\Lambda^k:
       \Theta\wedge\omega=0,\ d\Theta\wedge\omega=0\},
 \qquad k\ge n+1,\qquad d_c=d|_{J^k},
\end{equation}
where $\Lambda^k$ denotes smooth differential $k$-forms.
The two conditions imply $d\omega\in J^{k+1}$: differentiate
$\Theta\wedge\omega=0$ and $d\Theta\wedge\omega=0$.
Consequently $d_c^2=0$.
Since $2n-1\ge n+1$ for $n\ge2$, the degrees used here are all
in this range. In particular, for a scalar $a$ and
$\omega\in J^{2n-1}$,
\[
 d_c(a\omega)=da\wedge\omega+a\,d_c\omega.
\]

Every form in $J^{2n}$ is uniquely $\omega_V$ for a horizontal
vector field $V$. Indeed, $\Theta\wedge\omega=0$ writes it as
$\Theta\wedge\beta$ with $\beta$ horizontal of degree $2n-1$,
and contraction with $\Omega_H$ identifies such $\beta$ with $V$;
the condition involving $d\Theta$ is automatic in this degree.
Use the exterior-form norms induced by $\gOne$ on $J^{2n}$ and
$J^{2n-1}$; in particular, $\|\omega_V\|=|V|$.
The $C^0$ norm of a test form is the supremum of its pointwise norm.
For a current $T$ of locally
finite mass, its mass measure is determined by
\[
 \|T\|(U)=\sup\bigl\{|T(\omega)|:
 \omega\in C_c^\infty(U;J^{2n}),\ \|\omega(x)\|\le1\bigr\}.
\]
A horizontal vector $v$ determines the dual class
$\mathfrak t_x(v)$ by
\[
 \langle\mathfrak t_x(v),\omega_V(x)\rangle=\langle v,V(x)\rangle.
\]
Its norm is $|v|$.
Thus if $T=\tau\mu$, $|\tau|=1$ $\mu$-almost everywhere, then
$\|T\|=\mu$. The boundary is
$\partial_cT(\omega)=T(d_c\omega)$.
The term \emph{closed} means $\partial_cT=0$, an integration-by-parts
identity.

\begin{lemma}
\ifcsname r@weak:lem:bv-rumin\endcsname%
\bookmark[level=3,dest={\getrefbykeydefault{weak:lem:bv-rumin}{anchor}{}}]%
  {Lemma \ref*{weak:lem:bv-rumin}: Horizontal boundary currents of sets}%
\fi
\label{weak:lem:bv-rumin}
For a set $F$ of finite horizontal perimeter, let $\partial_H^*F$
denote its horizontal reduced boundary, which carries $\mu_F$
\cite{FSSC}. The current
\begin{equation}\label{eq:boundary-current-action}
 T_F(\omega_V)=\int_{\partial_H^*F}\langle\nu_F,V\rangle\dd\mu_F
             =\int_F\divh V\dd x
\end{equation}
has finite mass and is closed:
\begin{equation}\label{weak:legacy-eq-2-2}
 T_F(d_c\omega)=0
 \qquad\bigl(\omega\in C_c^\infty(\Hn;J^{2n-1})\bigr).
\end{equation}
Its orientation and mass are
\begin{equation}\label{weak:legacy-eq-2-3}
 T_F=\tau_F\mu_F,\qquad
 \tau_F(x)=\mathfrak t_x(\nu_F(x)),\qquad \|T_F\|=\mu_F.
\end{equation}
If $F$ also has finite Riemannian perimeter, the boundary integral
in \eqref{eq:boundary-current-action} may be taken over $\partial^*F$.
On smooth noncharacteristic boundary pieces, $\mu_F$ is horizontal
surface measure and the orientation is outward.
If $F$ is bounded, $T_F$ is compactly supported.
\end{lemma}
\begin{proof}
Cartan's formula and the zero Haar divergence of the frame fields give
$d\omega_V=(\divh V)\Theta\wedge\Omega_H$.
The definition $D_H\mathbf1_F=-\nu_F\mu_F$ gives
\eqref{eq:boundary-current-action}. For $d_c\omega=\omega_V$,
$d\omega_V=d^2\omega=0$, hence $\divh V=0$ and
\eqref{weak:legacy-eq-2-2} follows. Polar decomposition gives
\eqref{weak:legacy-eq-2-3}; when $F$ is bounded, its current has
compact support. On smooth pieces, Gauss--Green gives
$d\mu_F=|N_b|\,d\sigma_\Sigma$ and $\nu_F=N_b/|N_b|$.
\end{proof}

\subsection{Asymptotic closedness}
\label{sec:weighted-cycles}

\begin{lemma}
\ifcsname r@lem:centre-weight-error\endcsname%
\bookmark[level=3,dest={\getrefbykeydefault{lem:centre-weight-error}{anchor}{}}]%
  {Lemma \ref*{lem:centre-weight-error}: Vanishing error for centre-dependent weights}%
\fi
\label{lem:centre-weight-error}
Under \textup{(B1)--(B3)}, for every $\varphi\in C_c^1(\Hn)$
and $\omega\in C_c^\infty(\Hn;J^{2n-1})$,
\begin{equation}\label{eq:cycle-defect}
 \left|\int_{\Sigma_j}\varphi(\Ca_j)
       \langle\tau_{F_j},d_c\omega\rangle\dd\sigma_{H,j}\right|
 \le\|\omega\|_{C^0}\|D\varphi\|_{C^0}
       \int_{\Sigma_j}|D_H^{\Sigma_j}\Ca_j|\dd\sigma_{H,j}
 \longrightarrow0.
\end{equation}
The estimate applies to all sufficiently large $j$.
\end{lemma}
\begin{proof}
Fix $j$ with $\int_{\Sigma_j}|D_H^{\Sigma_j}\Ca_j|\dd\sigma_{H,j}<\infty$, by (B3).
For the cutoffs in (B2), $b=\eta_{j,\ell}\varphi(\Ca_j)\in C_c^1(\Sigma_j)$.
In finitely many regular charts, extend the
partition-of-unity pieces of $b$ constantly in a transverse coordinate
and cut off away from their surface supports. Their sum is an ambient
$C_c^1$ extension, still denoted by $b$. Closedness of $T_{F_j}$ from
Lemma~\ref{weak:lem:bv-rumin} and $C^1$ approximation give
$\int_{\Sigma_j} b\langle\tau_{F_j},d_c\omega\rangle\dd\sigma_{H,j}
 =-T_{F_j}(db\wedge\omega)$.
At a noncharacteristic point, choose an oriented orthonormal
horizontal basis with $E_{2n}=\nu_{F_j}$. Then
$E_1,\ldots,E_{2n-1}$ span $\xi\cap T\Sigma_j$, and
$\tau_{F_j}$ is represented by $T\wedge E_1\wedge\cdots\wedge E_{2n-1}$.
Write $\omega=\Theta\wedge\beta$ with $\beta$ horizontal, and put
\[
 \beta_k:=(-1)^{k-1}
 \beta(E_1,\ldots,\widehat{E_k},\ldots,E_{2n-1}),
 \qquad 1\le k\le2n-1,
\]
where the hat denotes omission. Since $\Theta(T)=1$,
$\Theta(E_k)=0$ and $\iota_T\beta=0$, expanding the wedge product gives
\[
 \begin{aligned}
 \langle\tau_{F_j},db\wedge\omega\rangle
 &=(db\wedge\Theta\wedge\beta)(T,E_1,\ldots,E_{2n-1})\\
 &=-(db\wedge\beta)(E_1,\ldots,E_{2n-1})\\
 &=-\sum_{k=1}^{2n-1}(E_kb)\beta_k.
 \end{aligned}
\]
The coefficients $\beta_k$ form a subset, up to signs, of the
coefficients of $\beta$ in an orthonormal exterior basis. Hence
$\sum_{k=1}^{2n-1}|\beta_k|^2\le|\beta|^2$.
Also $|\Theta\wedge\beta|=|\beta|$, since $\Theta$ is unit and
$\beta$ is horizontal. Cauchy--Schwarz therefore gives
\[
 \begin{aligned}
 |\langle\tau_{F_j},db\wedge\omega\rangle|
 &\le\left(\sum_{k=1}^{2n-1}|E_kb|^2\right)^{1/2}
       \left(\sum_{k=1}^{2n-1}|\beta_k|^2\right)^{1/2}\\
 &\le|D_H^{\Sigma_j}b|\,|\beta|
  =|D_H^{\Sigma_j}b|\,|\omega|.
 \end{aligned}
\]
Only derivatives along $\xi\cap T\Sigma_j$ occur, so the pairing
depends only on the surface trace of $b$.
The product and chain rules therefore give
\begin{equation}\label{eq:cutoff-weighted-cycle}
 \begin{aligned}
 &\left|\int_{\Sigma_j}\eta_{j,\ell}\varphi(\Ca_j)
       \langle\tau_{F_j},d_c\omega\rangle\dd\sigma_{H,j}\right|\\
 &\quad\le\|\omega\|_{C^0}\left(
 \|\varphi\|_\infty\int_{\Sigma_j}|D_H^{\Sigma_j}\eta_{j,\ell}|\dd\sigma_{H,j}
 +\|D\varphi\|_{C^0}\int_{\Sigma_j}\eta_{j,\ell}|D_H^{\Sigma_j}\Ca_j|\dd\sigma_{H,j}
 \right).
 \end{aligned}
\end{equation}
Let $\ell\to\infty$. The first term vanishes by (B2), and the
left-hand integrand is dominated by
$\|\varphi\|_\infty|d_c\omega|$, integrable against
$\sigma_{H,j}$. The second integral is bounded by
$\int_{\Sigma_j}|D_H^{\Sigma_j}\Ca_j|\dd\sigma_{H,j}$.
This proves the estimate in \eqref{eq:cycle-defect}; (B3) gives
its limit as $j\to\infty$.
\end{proof}

\section{Classification of minimizers and the sharp inequality}
\label{sec:one-centre-recognition}

\subsection{Constancy of closed subcurrents}

For a distribution $f$ on a connected open set, $D_Hf=0$ implies
$\partial_tf=-\tfrac12[X_1,Y_1]f=0$. The formulas for $X_k,Y_k$
then give all Euclidean derivatives zero, so $f$ is constant.

\begin{lemma}[Sobolev inequality on $\Hn$]
\ifcsname r@lem:profile-sobolev\endcsname%
\bookmark[level=3,dest={\getrefbykeydefault{lem:profile-sobolev}{anchor}{}}]%
  {Lemma \ref*{lem:profile-sobolev}: Sobolev inequality on Heisenberg groups}%
\fi\label{lem:profile-sobolev}
Let $E$ be a bounded global minimizer of horizontal perimeter at
its own positive volume. Put $p=Q/(Q-1)$ and
$\kappa=P_H(E)/|E|^{1/p}$. Then
\[
 \kappa\|h\|_{L^p(\Hn)}\le\int_{\Hn}|\nabla_Hh|\dd x
 \qquad\bigl(h\in C_c^\infty(\Hn)\bigr).
\]
\end{lemma}
\begin{proof}
The profile \eqref{sel:tag-2-11} gives
$P_H(F)\ge\kappa|F|^{1/p}$ for every finite-volume competitor $F$.
The coarea argument for the geometric Sobolev inequality
\cite[p.~267]{CDG1994} preserves the constant $\kappa$.
For $h\in C_c^\infty(\Hn)$, the layer-cake identity
$|h|=\int_0^\infty\mathbf1_{\{|h|>t\}}\dd t$
and Minkowski's integral inequality therefore give
\[
 \begin{aligned}
 \kappa\|h\|_{L^p}
 &\le\kappa\int_0^\infty|\{|h|>t\}|^{1/p}\dd t\\
 &\le\int_0^\infty P_H(\{|h|>t\})\dd t
 =\int_{\Hn}|\nabla_Hh|\dd x.
 \end{aligned}
\]
The last equality is weighted Euclidean coarea
\cite[Chapter~13]{Maggi}, with weight
$|\nabla_Hh|/|\nabla h|$ where $\nabla h\ne0$ and zero elsewhere.
Thus the Sobolev constant here is determined by $E$ through
\eqref{sel:tag-2-11}.

\end{proof}

\begin{remark}\label{rem:sobolev-noncircularity}
The coefficient $\kappa=P_H(E)/|E|^{(Q-1)/Q}$ in
Lemma~\ref{lem:profile-sobolev} is determined by the global minimizer $E$.
The profile \eqref{sel:tag-2-11} follows from global minimality and
Heisenberg dilations alone, and coarea transfers the same coefficient
to the Sobolev inequality. The identity $\kappa=C_n$ is a consequence
of the subsequent classification of minimizers, not an input to this
lemma or to Proposition~\ref{weak:prop:closed-subcurrent}.
\end{remark}

\begin{lemma}[Poincar\'e lemma for horizontal currents]
\ifcsname r@lem:horizontal-poincare\endcsname%
\bookmark[level=3,dest={\getrefbykeydefault{lem:horizontal-poincare}{anchor}{}}]%
  {Lemma \ref*{lem:horizontal-poincare}: Poincare lemma for horizontal currents}%
\fi\label{lem:horizontal-poincare}
Let $n\ge2$ and let $T$ be a compactly supported closed
upper-Rumin $2n$-current of finite mass on $\Hn$.
Write $T=\mathfrak t(\nu)\lambda$, where $\lambda=\|T\|$
and $\nu$ is a unit horizontal vector $\lambda$-almost everywhere.
There is a compactly supported distribution $f$ on $\Hn$
satisfying $D_Hf=-\nu\lambda$.
\end{lemma}
\begin{proof}
Let $\beta$ be the horizontal distributional one-form representing
$-\nu\lambda$. Rumin's contact complex has the same cohomology as
the de Rham complex \cite[p.~401]{Rumin1990}.
On horizontal one-forms, $d_c$ denotes the exterior derivative
modulo the differential ideal generated by $\Theta$ and $d\Theta$,
extended to distributional coefficients.
Let $\Pi_{\mathcal E}$ denote the differential projection onto the
subcomplex $\mathcal E$ of de Rham forms representing the Rumin complex.
The current comparison
\cite[Theorem~2.7, Lemma~4.8, and Proposition~4.9]{FranchiPansu}
gives $d_c\beta=0$ by closedness of $T$, and
$\alpha=\Pi_{\mathcal E}\beta$ is a closed distributional one-form
with $\alpha|_\xi=\beta|_\xi$.
Since $\Pi_{\mathcal E}$ is a differential operator,
$\spt\alpha\subset\spt\beta$, so $\alpha$ is compactly supported.

In Euclidean coordinates write
$\alpha=\sum_{i=1}^{2n+1}\alpha_i\,dx_i$ and let $\Gamma$ satisfy
$\Delta_{\mathrm E}\Gamma=\delta_0$. Ordinary Euclidean convolution
$*_{\mathrm E}$ defines
\[
 f=\sum_i(\partial_i\Gamma)*_{\mathrm E}\alpha_i.
\]
Each convolution is well defined because $\alpha_i$ has compact
support. Closedness gives $\partial_j\alpha_i=\partial_i\alpha_j$,
so differentiation of convolutions yields
\[
 \begin{aligned}
 \partial_jf
 &=\sum_i(\partial_i\Gamma)*_{\mathrm E}\partial_j\alpha_i\\
 &=\sum_i(\partial_i\Gamma)*_{\mathrm E}\partial_i\alpha_j\\
 &=\sum_i(\partial_i^2\Gamma)*_{\mathrm E}\alpha_j
   =(\Delta_{\mathrm E}\Gamma)*_{\mathrm E}\alpha_j=\alpha_j.
 \end{aligned}
\]
Thus $df=\alpha$. Since $\Theta$ vanishes on horizontal vectors,
the horizontal part of this identity is $D_Hf=-\nu\lambda$.
At this point $f$ is a distribution; no integrability has been assumed.

Choose a Euclidean ball containing $\spt\alpha$. On its connected
exterior, $df=0$, so $f$ is a constant distribution there.
Subtracting that constant makes $f$ compactly supported and
leaves its derivatives unchanged.
\end{proof}

\begin{proposition}[Constancy of closed subcurrents]
\ifcsname r@weak:prop:closed-subcurrent\endcsname%
\bookmark[level=3,dest={\getrefbykeydefault{weak:prop:closed-subcurrent}{anchor}{}}]%
  {Proposition \ref*{weak:prop:closed-subcurrent}: Constancy of closed subcurrents}%
\fi\label{weak:prop:closed-subcurrent}
Let $E$ be a bounded global minimizer of horizontal perimeter at
its own positive volume. If $T=\tau_E\lambda$ is a closed
upper-Rumin $2n$-current with $0\le\lambda\le\mu_E$, then
$T=cT_E$ and $\lambda=c\mu_E$ for some $c\in[0,1]$.
\end{proposition}
\begin{proof}
Since $0\le\lambda\le\mu_E$ and $E$ is bounded,
$\spt T\subset\spt\mu_E\subset\overline E$ is compact.
Lemma~\ref{lem:horizontal-poincare} gives a compactly supported
distribution $f$ with $D_Hf=-\nu_E\lambda$.
Put $p=Q/(Q-1)>1$ and $\kappa=P_H(E)/|E|^{1/p}$.
Using the nonnegative approximate
identity and group convolution of \eqref{eq:left-convolution}, set
\[
 f_\delta=\rho_\delta*f,\qquad
 g_\delta=\rho_\delta*\mathbf1_E-f_\delta.
\]
For small $\delta$, both functions are smooth and supported in a
common compact set. Left-invariance of the horizontal derivatives gives
\[
 \begin{aligned}
 D_Hf_\delta&=-\rho_\delta*(\nu_E\lambda),\\
 D_Hg_\delta&=-\rho_\delta*(\nu_E(\mu_E-\lambda)).
 \end{aligned}
\]
Here $\lambda$ and $\mu_E-\lambda$ are nonnegative, and
$|\nu_E|=1$ almost everywhere for both measures. Lemma~\ref{lem:profile-sobolev} and \eqref{eq:convolution-contraction} therefore imply
\[
 \begin{aligned}
 \kappa\|f_\delta\|_{L^p}
 &\le\int_{\Hn}|D_Hf_\delta|\dd x\le\lambda(\Hn),\\
 \kappa\|g_\delta\|_{L^p}
 &\le\int_{\Hn}|D_Hg_\delta|\dd x\le(\mu_E-\lambda)(\Hn).
 \end{aligned}
\]
Their sum is bounded by $P_H(E)$ because the two measures
add to $\mu_E$.

Since $1<p<\infty$, these bounds give a common subsequence along
which $f_\delta$ and $g_\delta$ converge weakly in $L^p$.
Convolution by an approximate identity also gives
$f_\delta\to f$ and $g_\delta\to\mathbf1_E-f$ in distributions.
The two kinds of convergence agree on smooth compactly supported
test functions, so their weak limits represent precisely $f$ and
$\mathbf1_E-f$. In particular, both are $L^p$ functions.
Weak lower semicontinuity of the two norms, followed by the
triangle inequality, gives
\[
 \|\mathbf1_E\|_{L^p}
 \le\|f\|_{L^p}+\|\mathbf1_E-f\|_{L^p}
 \le P_H(E)/\kappa=\|\mathbf1_E\|_{L^p}.
\]
Both inequalities are equalities. If either summand is zero,
then $f=0$ or $f=\mathbf1_E$. Otherwise, strict convexity of
$L^p$ implies
\[
 \frac{f}{\|f\|_{L^p}}
 =\frac{\mathbf1_E-f}{\|\mathbf1_E-f\|_{L^p}}.
\]
Adding the two proportional functions gives $f=c\mathbf1_E$, where
\[
 c=\frac{\|f\|_{L^p}}
 {\|f\|_{L^p}+\|\mathbf1_E-f\|_{L^p}}\in[0,1].
\]

Finally, comparing horizontal derivatives gives
$-\nu_E\lambda=D_Hf=cD_H\mathbf1_E=-c\nu_E\mu_E$.
Taking total variations yields $\lambda=c\mu_E$, and hence $T=cT_E$.
\end{proof}

\subsection{Identification of the Pansu ball}

\begin{proof}[Proof of Theorem~\ref{thm:weak-limit-characterization}]
Strict convergence in (B1) gives weak convergence
$\sigma_{H,j}\rightharpoonup\mu_E$ and convergence of their total
masses to $P_H(E)$. Thus the perimeter measures are tight.
The bounds following \eqref{eq:centres} place the centres of points
in any fixed compact set in another compact set, independently of
the unit normal. Consequently the centre measures are tight as well.
Along a subsequence,
\[
 \vartheta_j=(\Ca_j)_\#\sigma_{H,j}\rightharpoonup\vartheta,
 \qquad \vartheta(\Hn)=P_H(E)>0.
\]
Fix $a\in\spt\vartheta$ and $r>0$. Choose
$\varphi\in C_c^1(B_r(a))$, $0\le\varphi\le1$, with
$\vartheta(\varphi)>0$, where $B_r(a)$ is Euclidean.
After a further subsequence, the weighted measures and currents
converge weakly:
\[
 \varphi(\Ca_j)\sigma_{H,j}\rightharpoonup\lambda_\varphi,
 \qquad \varphi(\Ca_j)T_{F_j}\rightharpoonup T_\varphi.
\]
Strict convergence in (B1) gives
$\sigma_{H,j}\rightharpoonup\mu_E$ and $T_{F_j}\rightharpoonup T_E$.
The two nonnegative weights $\varphi(\Ca_j)$ and
$1-\varphi(\Ca_j)$ imply the measure inequalities
\[
 \|T_\varphi\|\le\lambda_\varphi,\qquad
 \|T_E-T_\varphi\|\le\mu_E-\lambda_\varphi.
\]
Using $\mathfrak t$ to identify current densities with horizontal
vectors, write their Radon--Nikodym representations as
\[
 \lambda_\varphi=a_\varphi\mu_E,\qquad
 T_\varphi=\mathfrak t(v_\varphi)\mu_E.
\]
The measure inequalities give, $\mu_E$-almost everywhere,
\[
 0\le a_\varphi\le1,\qquad
 |v_\varphi|\le a_\varphi,\qquad
 |\nu_E-v_\varphi|\le1-a_\varphi.
\]
Consequently,
\[
 1=|\nu_E|\le|v_\varphi|+|\nu_E-v_\varphi|
 \le a_\varphi+(1-a_\varphi)=1.
\]
Equality forces $|v_\varphi|=a_\varphi$ and the two vectors
$v_\varphi$ and $\nu_E-v_\varphi$ to be nonnegative multiples of
$\nu_E$. This includes the cases $a_\varphi=0$ and $a_\varphi=1$.
Thus $v_\varphi=a_\varphi\nu_E$ and
$T_\varphi=\tau_E\lambda_\varphi$.
Lemma~\ref{lem:centre-weight-error} makes $T_\varphi$ closed.
Proposition~\ref{weak:prop:closed-subcurrent} and total masses then give
\[
 \lambda_\varphi=c_\varphi\mu_E,
 \qquad c_\varphi=\frac{\vartheta(\varphi)}{P_H(E)}>0.
\]

By \eqref{eq:profile}, $x\in\PSphere_{\Ca_j(x)}$.
The weighted measures are therefore supported in the compact set
\[
 K_r=\overline B_r(a)\cdot\PSphere_0,
\]
and so is their weak limit
$c_\varphi\mu_E$. Since $c_\varphi>0$,
$\spt\mu_E\subset K_r$ for every $r>0$. If $x\in\bigcap_{r>0}K_r$,
choose $b_k\to a$ with $b_k^{-1}\cdot x\in\PSphere_0$;
closedness of $\PSphere_0$ gives $a^{-1}\cdot x\in\PSphere_0$.
Thus
\[
 \bigcap_{r>0}K_r=\PSphere_a,\qquad
 \spt\mu_E\subset\PSphere_a.
\]

It follows that $D_H\mathbf1_E=0$ in $\Omega_a$ and the connected
exterior of $\overline{\Omega_a}$, so $\mathbf1_E$ is constant on each.

By \eqref{eq:sharp-constant}, $\PSphere_0$ is the union of two
graphs over $\{|z|\le1\}$. Each vertical section contains at
most two points, so Fubini's theorem gives $|\PSphere_0|=0$.
Left translations preserve Haar measure, hence $|\PSphere_a|=0$.

Boundedness of $E$ makes the exterior constant zero. Since $|E|>0$
and $|\PSphere_a|=0$, the interior constant is one. Therefore
\[
 E=\Omega_a\quad\text{almost everywhere}.
\]

\end{proof}

\subsection{Proof of the sharp inequality}\label{sec:limit}

\begin{proof}[Proof of Theorem~\ref{thm:main}]
Normalize any global minimizer $E_0$ by \eqref{eq:target-scale}.
Proposition~\ref{prop:normalized-approximants} and
Theorem~\ref{thm:weak-limit-characterization} give
$\delta_{\rho_0}E_0=\Omega_a$
almost everywhere. Since dilations are group automorphisms,
\[
 E_0=(\delta_{\rho_0^{-1}}a)\cdot\delta_{\rho_0^{-1}}\Omega
 \quad\text{up to a null set}.
\]

Dilations preserve minimality, so for every $v>0$, \eqref{sel:tag-2-11} becomes
\begin{equation}\label{eq:sharp-profile}
 I(v)=P_H(\Omega)\left(\frac{v}{|\Omega|}\right)^{(Q-1)/Q}
 =C_n v^{(Q-1)/Q}.
\end{equation}
For a finite-perimeter set $F$ with $0<|F|<\infty$, $P_H(F)\ge I(|F|)$ gives
\eqref{eq:sharp-isoperimetric}. Translates and dilates of $\Omega$
attain equality, so $C_n$ is optimal. Conversely, equality makes $F$ a global minimizer,
so the classification above applies.
\end{proof}

\begin{corollary}[Convergence of reconstructed centres]
\ifcsname r@cor:centre-convergence\endcsname%
\bookmark[level=3,dest={\getrefbykeydefault{cor:centre-convergence}{anchor}{}}]%
  {Corollary \ref*{cor:centre-convergence}: Convergence of reconstructed centres}%
\fi
\label{cor:centre-convergence}
Under the hypotheses of Theorem~\ref{thm:weak-limit-characterization},
suppose the $F_j$ lie in a common compact set and write
$E=\Omega_a$ almost everywhere. Then
\[
 (\Ca_j)_\#\sigma_{H,j}\rightharpoonup P_H(E)\delta_a,
 \qquad
 \int_{\Sigma_j}|\Ca_j-a|\dd\sigma_{H,j}\longrightarrow0.
\]
\end{corollary}
\begin{proof}
The bounds after \eqref{eq:centres} put all centre measures in a
common compact set. The proof of
Theorem~\ref{thm:weak-limit-characterization} gives
$E=\Omega_b$ for every $b$ in the support of any subsequential
limit. The symmetries of \eqref{eq:sharp-constant} and the group law give
$\int_{\Omega_b}x\,dx=|\Omega|b$, hence $b=a$.
Convergence of total masses therefore identifies every weak limit
as $P_H(E)\delta_a$. Testing this convergence against a compactly
supported continuous function equal to $|b-a|$ on the common
centre support proves the integral limit.
\end{proof}

\bookmarksetup{depth=-1}

\begingroup
\footnotesize\linespread{1.0}\selectfont
\bigskip
\noindent Department of Mathematics, National Taiwan University, Taipei 10617, Taiwan\\
\textit{Email address:} \href{mailto:b14201028@ntu.edu.tw}{b14201028@ntu.edu.tw}

\medskip
\noindent Institute of Mathematics, Academia Sinica, Taipei 106319, Taiwan\\
\textit{Email address:} \href{mailto:wtkao@as.edu.tw}{wtkao@as.edu.tw}

\medskip
\noindent Department of Mathematics, National Taiwan University, Taipei 10617, Taiwan\\
\textit{Email address:} \href{mailto:d13221002@ntu.edu.tw}{d13221002@ntu.edu.tw}
\endgroup
\end{document}